\documentclass{conm-p-l}
\usepackage{verbatim}
\usepackage{amssymb}
\usepackage{amsbsy}
\usepackage{amscd}
\usepackage{amsmath}
\usepackage{amsthm}
\usepackage{amsxtra}
\usepackage{latexsym}
\usepackage[mathscr]{eucal}

\newcommand{\FFcenter}{\mf{z}(\affg)}

\newcommand{\wSl}{\tilde{\mathcal{S}}}
\newcommand{\D}{\mc{D}}
\newcommand{\betagamma}[1]{\mc{D}^{\ch}(\C^{#1})}
\newcommand{\resV}{V_{\res}(\fing)}
\newcommand{\resW}{\W_{\res}(\fing,f)}
\newcommand{\resWa}{M_{\fing}}
\newcommand{\criV}[1]{V^{\cri}_{#1}(\fing)}
\newcommand{\criW}[1]{\W^{\cri}_{#1}(\fing,f)}
\newcommand{\cri}{\on{cri}}

\newcommand{\M}[1]{\mathbf{M}_{#1}}
\newcommand{\Wak}[1]{\mathbf{F}_{#1}}

\newcommand{\Sl}{\mathbb{S}}

\newcommand{\wt}{\widetilde}

\newcommand{\Nil}{\mc{N}}

\newcommand{\BRS}[2]{H^{\semiinf+#1}_{f}(#2)}
\newcommand{\bw}[1]{\bigwedge\nolimits^{#1}}

\newcommand{\wh}{\widehat}

\newcommand{\mc}{\mathcal}
\newcommand{\mf}{\mathfrak}

\newcommand{\on}{\operatorname}
\newcommand{\affP}{\widehat{P}_k}

\newcommand{\Vg}[1]{V^{#1}(\fing)}

\newcommand{\finn}{\mathfrak{n}}

\newcommand{\isomap}{{\;\stackrel{_\sim}{\to}\;}}

\newcommand{\W}{\mathscr{W}}

\newcommand{\nc}{\newcommand}
\nc{\Hp}[1]{H^{#1}}
\newcommand{\prin}{\mathrm{prin}}

\newcommand{\h}{\mathfrak{h}}

\newcommand{\affh}{\widehat{\mathfrak{h}}}

\newcommand{\affg}{\widehat{\mathfrak{g}}}
\newcommand{\bigaffg}{\widetilde{\mathfrak{g}}}
\newcommand{\fing}{\mathfrak{g}}
\newcommand{\finh}{\mathfrak{h}}
\newcommand{\finm}{\mathfrak{m}}

\newcommand{\Wg}[1]{\W^{#1}(\fing, f)}

\newcommand{\N}{\mathbb{N}}
\newcommand{\Q}{\mathbb{Q}}
\newcommand{\1}{{\mathbf{1}}}

\newcommand{\dual}[1]{{#1}^*}
\newcommand{\bra}{{\langle}}
\newcommand{\ket}{{\rangle}}

\newcommand{\Lam}{\Lambda}

\newcommand{\lam}{\lambda}
\newcommand{\ra}{\rightarrow}
\newcommand{\+}{\mathop{\oplus}}
\newcommand{\Z}{\mathbb{Z}}

\newcommand{\inv}{^{-1}}

\renewcommand{\*}{{\otimes}}
\newcommand{\C}{\mathbb{C}}
\newcommand{\che}{^{\vee}}

\theoremstyle{plain}
\newtheorem{Th}{Theorem}[section]

\newtheorem{Pro}[Th]{Proposition}

\theoremstyle{definition}

\theoremstyle{remark}

\newtheorem{Rem}[Th]{Remark}

\newcommand{\semiinf}{\frac{\infty}{2}}
\DeclareMathOperator{\res}{res}

\DeclareMathOperator{\Zhu}{Zhu}

\DeclareMathOperator{\rank}{rk}
\DeclareMathOperator{\ch}{ch}

\DeclareMathOperator{\gr}{gr}

\DeclareMathOperator{\ad}{ad}
\DeclareMathOperator{\Ad}{Ad}

\DeclareMathOperator{\Spec}{Spec}

\title{ W-algebras at the critical level
}
\author{Tomoyuki Arakawa}
\address{Research Institute for Mathematical Sciences, Kyoto University,
 Kyoto 606-8502 JAPAN}

\email{arakawa@kurims.kyoto-u.ac.jp}

\thanks{This work is partially  supported 
by the JSPS Grant-in-Aid  for Scientific Research (B)
No.\ 20340007
and the JSPS Grant-in-Aid for challenging Exploratory Research
No.\ 23654006}


\begin{document}
\maketitle

\begin{abstract}
Let $\fing$ be a complex simple Lie algebra,
$f$ a nilpotent element of $\fing$.
We show that 
(1) the center 
of the   $W$-algebra 
$\Wg{\cri}$
associated $(\fing,f)$
 at the critical  level
coincides with the Feigin-Frenkel center of $\affg$,
(2) 
the 
centerless quotient
$\W_\chi(\fing,f)$
 of $\W^{\cri}(\fing,f)$ 
corresponding
to  an ${}^L\fing$-oper $\chi$
on the disc
is simple,
and (3)
the simple quotient  $\W_\chi(\fing,f)$
is a  quantization
of the jet scheme of the intersection of the Slodowy slice 
at $f$ with the
 nilpotent cone of $\fing$.
\end{abstract}

\section{Introduction}
Let $\fing$ be a complex simple Lie algebra,
$f$ a nilpotent element of $\fing$,
$U(\fing,f)$
the finite $W$-algebra \cite{Pre02}
associated with $(\fing,f)$.
In \cite{Pre07}
it was shown that
the center of $U(\fing,f)$ coincides with 
the center $\mc{Z}(\fing)$ of the universal enveloping algebra
$U(\fing)$ of $\fing$
(Premet attributes the proof to Ginzburg).

Let $\W^k(\fing,f)$
 be the 
{\em (affine) $W$-algebra}
 \cite{FF90,KacRoaWak03,KacWak04}
 associated with $(\fing,f)$ at level $k\in \C$.
One may \cite{Ara07,De-Kac06} regard 
$\W^k(\fing,f)$ as a
one-parameter  {\em chiralization} of 
$U(\fing,f)$.
Hence
it is natural to ask 
whether the analogous identity holds for the
center 
$Z(\W^k(\fing,f))$
of $\W^{k}(\fing,f)$,
which is a commutative vertex subalgebra of 
$\W^{k}(\fing,f)$.

Let $\Vg{k}$
be the universal affine vertex algebra associated with $\fing$
at level $k$,
$Z(\Vg{k})$ the center of $\Vg{k}$.
The embedding
$Z(V^{k}(\fing))
\hookrightarrow V^{k}(\fing)$
induces    the vertex algebra homomorphism
\begin{align*}
Z(V^{k}(\fing))\ra Z(\W^{k}(\fing,f))
\end{align*}
for any $k\in \C$.
However, 
both 
$Z(V^{k}(\fing))$
and 
$Z(\W^{k}(\fing,f))
$
are trivial
unless
$k$ is the critical level 
\begin{align*}
\cri:=-h\che,
\end{align*}
where $h\che$ is the dual Coxeter number of $\fing$.
Therefore the question one should ask is that 
whether
the center 
$Z(\W^{\cri}(\fing,f))
$ of the $W$-algebra 
at the critical level 
coincides with
the {\em Feigin-Frenkel center}
\cite{FeiFre92,Fre05}
$\mf{z}(\affg):=Z(V^{\cri}(\fing))$,
which can be naturally considered as 
the space of functions on the space
of $\on{Op}_{{}^L\fing}^{\on{reg}}$
of ${}^L\fing$-opers on the disc.
Here ${}^L\fing$ is the Langlands dual Lie algebra of $\fing$.

\begin{Th}\label{MainTh}
 The embedding
$\FFcenter
\hookrightarrow V^{\cri}(\fing)$
induces the isomorphism
\begin{align*}
\FFcenter\isomap  Z(\W^{\cri}(\fing,f)).
\end{align*}
Moreover, 
$\Wg{\cri}$ is free over $\FFcenter$,
where $\FFcenter$ is regarded as a commutative ring
with the $(-1)$-product.
\end{Th}

Theorem \ref{MainTh} generalizes 
a result of Feigin and Frenkel \cite{FeiFre92},
who proved  that
$\FFcenter\cong   \W^{\cri}(\fing,f_{\prin})$
for a principal nilpotent element $f_{\prin}$ of $\fing$.
It also generalizes 
a result of 
Frenkel and Gaitsgory \cite{FreGai04},
who proved the freeness of 
$\Vg{\cri}$ over $\FFcenter$.

\smallskip

Let $G$ be the adjoint group of $\fing$,
 $\Sl$  the Slodowy slice at $f$ to $\Ad G.f$,
$\mc{N}$ the nilpotent cone of $\fing$.
Set
\begin{align*}
 \mc{S}=\Sl\cap \Nil.
\end{align*}
It is known  \cite{Pre02}
that the scheme
$\mc{S}$ is reduced, irreducible,
 and  normal
complete intersection
of dimension $\dim \mc{N}-\dim \Ad G.f$.

For  $\chi\in \on{Op}_{{}^L\fing}^{\on{reg}}$,
let 
  $\criW{\chi}$
be the quotient of 
$\W^{\cri}(\fing,f)$ 
by the ideal generated by 
$z-\chi(z)$ with $ z\in \FFcenter$.
Then any simple quotient of $\Wg{\cri}$ 
is a quotient of
$\criW{\chi}$ for some $\chi$.
\begin{Th}\label{Th: main2}
For  $\chi\in \on{Op}_{{}^L\fing}^{\on{reg}}$,
the vertex algebra 
$\W_{\chi}(\fing,f)$ is simple.
Its associated graded vertex Poisson algebra
$\gr \W_{\chi}(\fing,f)$ is isomorphic to
$\C[\mc{S}_{\infty}]$
as vertex Poisson algebras,
where
 $\mc{S}_{\infty}$ is the infinite jet scheme of $\mc{S}$
and $\C[\mc{S}_{\infty}]$ is equipped with the level $0$
vertex Poisson algebra structure.
\end{Th}

Theorem \ref{Th: main2} generalizes  a result of 
Frenkel and Gaitsgory \cite{FreGai04},
who proved the simplicity
of
 the quotient of $\Vg{\cri}$
by the ideal generated by
$z-\chi(z)$ for $z\in \mf{z}(\affg)$.

In the case that $f=f_{\prin}$
we have
$\W_{\chi}(\fing,f_{\prin})=\C$ \cite{FeiFre92},
while
$\mc{S}$
is a point, and so is $\mc{S}_{\infty}$.
Theorem \ref{Th: main2}
 implies that
this is the only case that
$\W^{\cri}(\fing,f)$ admits
finite-dimensional
quotients.

\smallskip

In general
little  is known about the 
representations of $\Wg{\cri}$.
We have shown
in   \cite{Ara08-a} 
that
at least in type $A$
the representation theory of $\Wg{k}$ is controlled
by that of $\affg$ at level $k$
for any $k\in \C$.
Therefore 
the Feigin-Frenkel conjecture
(see  \cite{AraFie08})
implies
 that,
at least in type $A$,
the 
character of irreducible 
highest weight representations
of $\Wg{\cri}$
should be expressed in terms of
Lusztig's periodic polynomial \cite{Lus80}.
We plan to  return to this in future work.

\section{Associated graded vertex Poisson algebras}

For a vertex algebra $V$,
let $\{F^p V\}$ be the Li filtration \cite{Li05},
\begin{align*}
\gr V=\bigoplus_{p} F^p V/F^{p+1}V
\end{align*}
the associated graded vertex
Poisson  algebra.
The vertex Poisson algebra structure of $\gr V$
restricts to the Poisson algebra structure 
on  {\em Zhu's Poisson algebra} \cite{Zhu96}
\begin{align*}
R_V:=V/F^1 V\subset \gr V.
\end{align*}
Moreover there is a surjective map
\begin{align}
(R_V)_{\infty}\ra \gr V
\label{eq:homo}
\end{align}
of vertex Poisson algebras \cite{Li05,Ara09a}.
Here
$X_V=\Spec R_V$,
$(R_V)_{\infty}=\C[(X_V)_{\infty}]$,
where 
 $X_{\infty}$ denotes the infinite jet scheme of a
scheme $X$ of finite type,
and $(R_V)_{\infty}$ is equipped with the level zero vertex Poisson algebra
structure \cite[2.3]{Ara09a}.

\smallskip

Let $\betagamma{r}$
be the
 {\em  $\beta \gamma$-system}
 of rank $r$,
that is,
 the vertex algebra
generated by fields
$a_1(z),\dots,a_r(z)$,
$a_1^*(z),\dots,a_r^*(z)$,
satisfying the following OPE's:
\begin{align*}
 a_i(z)a_j(z)^*\sim \frac{\delta_{ij}}{z-w},
\quad
 a_i(z)a_j(z)\sim  a_i^*(z)a_j^*(z)\sim 0.
\end{align*}
It is straightforward to see that
$ R_{\betagamma{r}}
\cong \C[T^* \C^r]$ as 
 Poisson algebras
 and that
(\ref{eq:homo}) gives the isomorphism
\begin{align}
 (R_{\betagamma{r}})_{\infty}\isomap \gr \betagamma{r}.
\end{align}

\smallskip

Let $\fing$,
$f$ be as in Introduction,
$\rank \fing$ the rank of $\fing$,
$(~|~)$
the normalized invariant bilinear form of $\fing$.
Let $\mf{s}=\{e,h,f\}$ be an $\mf{sl}_2$-triple in $\fing$,
and let
$\fing_j=\{x\in \fing; [h,x]=2j x\}$ so that
\begin{align}
\fing=\bigoplus_{j\in \frac{1}{2}\Z}\fing_j.
\label{eq:good-grading}
\end{align}
Fix a triangular 
decomposition $\fing=\finn_-\+ \finh\+\finn$
such that
$h\in \finh\subset \fing_0$
and $\finn\subset \fing_{\geq 0}:=\bigoplus_{j\geq 0}
\fing_j$.
We will identify $\fing$ with $\fing^*$ via $(~|~)$.

The {\em Slodowy slice}
to $\Ad G.f$ at $f$
is by definition
the affine subspace
\begin{align*}
\Sl=f+\fing^e
\end{align*}
of $\fing$,
where $\fing^e$ is the centralizer of $e$ in $\fing$.
It is known \cite{GanGin02} that
the  Kirillov-Kostant
Poisson structure
of $\fing^*=\fing$ restricts to $\Sl$.

Let $\ell$ be an $\ad \h$-stable Lagrangian subspace
of $\fing_{1/2}$
with respect to the 
symplectic form
$\fing_{1/2}\times \fing_{1/2}\ra \C$,
$(x,y)\mapsto (f|[x,y])$.
Set
\begin{align*}
\finm=\ell\+\bigoplus_{j\geq 1}\fing_j,
\end{align*}
and 
let
$M$ be the unipotent subgroup of $G$ whose Lie algebra  is $\finm$, 
$\finm^{\bot}=\{x\in \fing; (x|y)=0\text{ for all }y\in
\finm\}$.
Then \cite{GanGin02}
we have the isomorphism of affine varieties 
\begin{align}
 M\times \Sl\isomap f+\finm^{\bot},
\quad (g,x)\mapsto \Ad(g)(x). 
\label{eq:Kostant-Gan-GInzburg}
\end{align}
This induces the following isomorphism
of jet schemes:
\begin{align}
 M_{\infty}\times \Sl_{\infty}\isomap 
(f+\finm^{\bot})_{\infty}.
\label{eq:Kostant-Gan-GInzburg-jets}
\end{align}
Denote by
$I$ and $I_{\infty}$  the defining ideals
of $f+\finm^{\bot}$ 
and $(f+\finm^{\bot})_{\infty}$
in $\fing$ and $\fing_{\infty}$,
respectively.
By (\ref{eq:Kostant-Gan-GInzburg})
and (\ref{eq:Kostant-Gan-GInzburg-jets})
we have
\begin{align*}
\C[\Sl]\cong (\C[\fing]/I)^{M},\quad
\C[\Sl_{\infty}]\cong (\C[\fing_{\infty}]/I_{\infty})^{M_{\infty}}.
\end{align*}

 Let 
\begin{align*}
\bigaffg=\fing[t,t\inv]\+ \C K\+ \C D
\end{align*}
be the affine Kac-Moody algebra
associated with $\fing$,
where $K$ is the central element and $D$ is the degree operator.
Set
$\affg=\fing[t,t\inv]\+ \C K$, the derived algebra of $\bigaffg$.

The {\em universal affine vertex algebra}
$\Vg{k}$
associated with $\fing$ at level $k\in \C$
is 
the induced $\affg$-module
$U(\affg)\*_{U(\fing[t]\+\C K)}\C_k$,
equipped with the natural vertex algebra structure (see e.g.\ \cite{Kac98,FreBen04}).
Here $\C_k$ is the one-dimensional
representation of 
$\fing[t]\+ \C K$ on which 
$\fing[t]$ acts trivially and $K$ acts as a multiplication by $k$.
 The Li filtration of $\Vg{k}$
is essentially the same 
as the
standard filtration of  $U(\fing[t\inv]t\inv)$
under the isomorphism
$U(\fing[t\inv]t\inv)\cong \Vg{k}$, see \cite{Ara09a}.
We have
\begin{align}
\label{eq:R-of-affine}
 &R_{\Vg{k}}\cong \C[\fing^*]
\end{align}
and (\ref{eq:homo})
gives the isomorphism
\begin{align}
\C[\fing^*_{\infty}]\isomap  \gr \Vg{k}
\label{eq:associated graded of affine}
\end{align}
of vertex Poisson algebras.
Let
$ \mf{z}(\affg)$ be the Feigin-Frenkel center
$Z(V^{\cri}(\fing))$ as in Introduction.
It is known 
 \cite{FeiFre92,Fre05,Fre07}
that
the Li filtration 
of $\Vg{\cri}$
restricts to the Li filtration of $\mf{z}(\affg)$.
Moreover
we have
\begin{align}
 R_{\mf{z}(\affg)}\cong \C[\fing^*]^G,
\label{eq:poisson-of-FF-center}
\end{align}
and (\ref{eq:homo})
gives the isomorphism 
\begin{align}
(R_{\mf{z}(\affg)})_{\infty}\isomap \gr \mf{z}(\affg).
\label{eq:associated graded of FF center}
\end{align}
(Hence
the vertex Poisson algebra structure of $\mf{z}(\affg)$ is
trivial\footnote{The vertex Poisson algebra
structure considered in this article is different from the
one in \cite{Fre07}}.)
The isomorphisms
(\ref{eq:poisson-of-FF-center}) and  (\ref{eq:associated graded of FF center})
imply  
 \cite{BeiDri,EisFre01}
that
\begin{align*}
 \gr \mf{z}(\affg)\cong 
\C[(\fing^*/\!/G)_{\infty}]=\C[\fing^*_{\infty}]^{G_{\infty}}.
\end{align*}

For 
$\chi \in \on{Op}_{{}^L\fing}^{\on{reg}}$,
let $\criV{\chi}$ be 
the quotient of $\Vg{\cri}$
by the ideal generated by
$z-\chi(z)$ for $z\in \mf{z}(\affg)$.
Because  
$\mf{z}(\affg)$ acts freely on 
$\Vg{\cri}$  \cite{FreGai04},
it follows from 
(\ref{eq:associated graded of affine})
and (\ref{eq:associated graded of FF center})
that
\begin{align}
 &R_{\criV{\chi}}\cong \C[\mc{N}],
\quad \gr \criV{\chi}\cong \C[\mc{N}_{\infty}].
\label{eq:assocaited-graded-of-critical-simple}
\end{align}
Furthermore, it was proved in 
\cite{FreGai04} that 
the vertex algebra
$\criV{\chi}$ is {\em simple} (thus in particular 
$\criV{\chi}$  is simple  as a $\affg$-module).

Let 
\begin{align*}
\chi_0\in \on{Op}_{{}^L\fing}^{\on{reg}}
\end{align*}
be 
the unique element 
such that
$\{z-\chi_0(z); z\in \FFcenter\}$
is the argumentation ideal 
$\mf{z}(\affg)^*$ of $\mf{z}(\affg)$.
We set
\begin{align*}
\resV=\criV{\chi_0},
\end{align*}
and call it the {\em restricted affine vertex algebra}
associated with $\fing$.
As a $\affg$-module,
$\resV$ is isomorphic to the irreducible highest weight
representation with highest weight $-h\che\Lam_0$.

\smallskip
For $k\in \C$ and
 a $\Vg{k}$-module $M$,
one can define
the
complex  $(C(M),d)$ of 
the BRST cohomology of
the generalized quantized Drinfeld-Sokolov reduction
associated with $(\fing,f)$ (\cite{KacRoaWak03}).
We have
$C(M)=M\* \betagamma{m}\* \bw{\semiinf+\bullet}$,
where
$
m=\frac{1}{2}\dim\fing_{\frac{1}
2}
$,
and 
$\bw{\semiinf+\bullet}$ is the Clifford vertex superalgebra 
of rank $\dim \finn$.
The complex 
$(C(M),d)$ can be identified with 
 Feigin's complex which defines
the semi-infinite cohomology $H^{\semiinf+\bullet}(\fing_{>0}[t,t\inv],
M\* \D^{\ch}(\C^m)
)$,
where 
$\fing_{>0}[t,t\inv]$-module structure of 
$\D^{\ch}(\C^m)$ is described in 
 \cite[\S 3]{Ara05}\footnote{In \cite{Ara05}
$\D^{\ch}(\C^m)$ is denoted by $\mc{F}^{\on{ne}}(\chi)$
}.
Let   
\begin{align*}
\BRS{\bullet}{M}:=H^{\bullet}(C(M),d).
\end{align*}
The {\em $W$-algebra associated with $(\fing,f)$ at level $k$}
is by
 definition
\begin{align*}
\Wg{k}=\BRS{0}{\Vg{k}},
\end{align*}
which is naturally a 
 vertex algebra because
$d$ is the zero mode of a odd  field $d(z)$
of $C(\Vg{k})$.
We have \cite{De-Kac06}
that, for any $k$,
\begin{align}
 R_{\Wg{k}}\cong \C[\Sl],
\label{eq:ring-of-W}
\end{align}
and (\ref{eq:homo})
gives the isomorphism 
\begin{align}
 \C[\Sl_{\infty}]\isomap \gr \Wg{k},
\label{eq:ring-of-grW}
\end{align}
see \cite{Ara09b}.

Let
$k=\cri$.
 For $z\in \mf{z}(\affg)$,
 we have 
$d z=0$, and the class of 
$z$ belongs to the center $Z(\Wg{\cri})$
 of $\Wg{k}$.
Hence
the embedding
$\mf{z}(\affg)
\hookrightarrow V^{\cri}(\fing)$
induces the vertex algebra homomorphism 
$\mf{z}(\affg)\ra   Z(\W^{cri}(\fing,f))\subset \Wg{\cri}$.

\begin{Pro}\label{Pro:injectivity}
  The embedding
$\mf{z}(\affg)
\hookrightarrow V^{\cri}(\fing)$
induces the embedding
$\mf{z}(\affg)\hookrightarrow   \W^{cri}(\fing,f)$.
\end{Pro}
\begin{proof}
It is sufficient to show that it induces an injective homomorphism
 $\gr \mf{z}(\affg)
\ra \gr \Wg{\cri}
$.
Under the identification
(\ref{eq:R-of-affine}) and (\ref{eq:ring-of-W}),
the induced map 
$R_{\Vg{\cri}}\ra R_{\Wg{\cri}}$
is identified  the restriction map
$\C[\fing^*]^G\ra \C[\Sl]$,
and hence is
 injective \cite{Kos78,Pre02}.
Therefore
it induces the injective map
$(R_{\Vg{\cri}})_{\infty}\hookrightarrow (R_{\Wg{\cri}})_{\infty}$,
which
is identical to 
the map $\gr \mf{z}(\affg)
\ra \gr \Wg{\cri}
$.
\end{proof}
By 
Proposition  \ref{Pro:injectivity}
we can define the quotient 
$\criW{\chi}$
of $\Wg{\cri}$ for $\chi\in 
\on{Op}_{{}^L\fing}^{\on{reg}}$
as in Introduction.
Let
\begin{align*}
\resW=\criW{\chi_0}
\end{align*}
and call it the {\em restricted $W$-algebra}
associated with $(\fing,f)$.
It is a graded quotient of $\Wg{\cri}$.

\begin{Rem}
 Let $\Zhu(\resW)$ 
be the Ramond twisted Zhu algebra \cite{De-Kac06}
of $\resW$.
Then from Proposition \ref{Pro:description-of-restricted-W}
below it follows that
\begin{align*}
\Zhu(\resW)\cong U(\fing,f)/Z(\fing)^* U(\fing,f),
\end{align*}
where $Z(\fing)^*$ is the argumentation ideal of $Z(\fing)$.
\end{Rem}
Set $\mc{S}=\Sl\cap \Nil$ as in Introduction.
By restricting (\ref{eq:Kostant-Gan-GInzburg})
and (\ref{eq:Kostant-Gan-GInzburg-jets}),
we obtain the isomorphisms
\begin{align}
&  M\times \mc{S}\isomap (f+\finm^{\bot})\cap \Nil,\\
 & M_{\infty}\times \mc{S}_{\infty}\isomap 
((f+\finm^{\bot})\cap \Nil)_{\infty}
=(f+\finm^{\bot})_{\infty}\cap \Nil_{\infty}.
\label{eq:jets-of-S-N}
\end{align}
\begin{Pro}\label{Pro:description-of-restricted-W}
We have the following.
\begin{enumerate}
 \item $\BRS{i}{\resV}=0$ for $i\ne 0$
and  
$\BRS{0}{\resV}\cong \resW$ as vertex algebras. 
\item $\Wg{\cri}$ is free over $\FFcenter$.
\item $R_{\resW}\cong \C[\mc{S}]$ as Poisson algebras
and $\gr \resW\cong \C[\mc{S}_{\infty}]$ as vertex Poisson algebras.
\end{enumerate}
\end{Pro}
\begin{proof}
(i) Because $\resV$ is $G$-integrable,
the vanishing assertion was already proved in   \cite{Ara09b}.
Let
$\Gamma^p\Vg{\cri}=
(\mf{z}(\affg)^*)^p \Vg{\cri}$.
Then
$\{\Gamma^p \Vg{\cri}\}$ defines a decreasing filtration
of $\Vg{\cri}$ as $\affg$-modules,
and the freeness of $\Vg{\cri}$ over $\mf{z}(\affg)$
implies that
\begin{align*}
 \gr^{\Gamma}\Vg{\cri}\cong \resV\*_{\C} \mf{z}(\affg)
\end{align*}
as $\Vg{\cri}\*_{\C} \mf{z}(\affg)$-modules.
The vanishing
of $\BRS{\bullet}{\resV}$
implies that the spectral sequence associated with  the filtration
$\{\Gamma^p \Vg{\cri}\}$ collapses at $E_1=E_{\infty}$
and that
 \begin{align}
  \gr^{\Gamma}\BRS{0}{\Vg{\cri}}\cong 
\BRS{0}{\resV}\*_{\C} \mf{z}(\affg).
\label{eq:freeness}
 \end{align}
as $\Vg{\cri}\*_{\C} \mf{z}(\affg)$-modules.
This proves
the second assertion.
(ii) follows from (\ref{eq:freeness}).
(iii)
By (i),
the  Li filtration of $\resV$ induces a filtration 
of $\resW=\BRS{0}{\resV}$,
which we temporary denote
 by  $\{G^p \resW\}$.
By
\cite[Theorem 4.4.3]{Ara09b}, 
(\ref{eq:assocaited-graded-of-critical-simple})
and (\ref{eq:jets-of-S-N})
we obtain
\begin{align*}
\gr^{G} \resW
\cong (\gr \resV /I_{\infty}\resV )^{M_{\infty}}
\cong (\C[\Nil_{\infty}]/I_{\infty}\C[\Nil_{\infty}])^{M_{\infty}}
= \C[\mc{S}_{\infty}].
\end{align*}
It remains to show that
the filtration $\{G^p \resW\}$ coincides with the
Li filtration of $\resW$.
But this can be seen as in the same manner as \cite[Theorem 4.4.6]{Ara09b}.
\end{proof}

\begin{Rem}
By the vanishing result of \cite{Gin08},
the same argument proves the freeness of
$U(\fing,f)$ over $\mc{Z}(\fing)$.
\end{Rem}
\begin{Pro}\label{Pro:conforma-weight-filtartion}
 Let $\lam \in \on{Op}_{{}^L\fing}^{\on{reg}}$.
Then $\BRS{i}{\criV{\chi}}=0$ for $i\ne 0$
and $\BRS{0}{\criV{\chi}}\cong \criW{\chi}$.
We have $\gr \criW{\chi}\cong \C[\mc{S}_{\infty}]$
as vertex Poisson algebras.
\end{Pro}
\begin{proof}
We have proved the assertion in the case that  $\chi=\chi_0$
in Proposition \ref{Pro:description-of-restricted-W}.
The general case follows from this by the following argument:
Consider the {\em conformal weight filtration}
of a vertex algebra $V$ defined in \cite[3.3.2]{AraCheMal08},
which we denote by $\{E_p V\}$.
Then $\gr^E \criV{\chi}$
is a $\affg$-module isomorphic to
$\resV$ for any $\lam \in \on{Op}_{{}^L\fing}^{\on{reg}}$.
Hence 
Proposition \ref{Pro:description-of-restricted-W}
implies that 
\begin{align}
\gr^E \BRS{i}{\criV{\chi}}=\begin{cases}
			     \resW&\text{for }i=0\\
0&\text{for }i \ne 0.	    \end{cases}
\label{eq:conformal-filtration--}
\end{align}
This completes the proof.
\end{proof}
\section{BRST cohomology of restricted Wakimoto modules}
Let $\affh=\finh\+ \C K\+ \C D$,
the Cartan subalgebra of $\bigaffg$,
$\dual{\affh}=\dual{\finh}\+ \C \delta\+ \C \Lam_0$
the dual of $\affh$,
where
$\delta$ and $\Lam_0$ are the elements dual to $D$ and $K$,
respectively.
Put
$\dual{\affh}_k=\{\lam\in \dual{\affh};
\lam(K)=k\}$.

For
$\lam\in \dual{\affh}_{\cri}$,
let $\Wak{\lam}^{\res}$ 
be the {\em restricted Wakimoto module}  
  \cite{FeuFre90,Fre05}\footnote{In \cite{Fre05}
$\Wak{\lam}^{\res}$ is denoted by $W_{\bar \lam/t}$.}
with highest weight $\lam$.
Set
\begin{align*}
 \resWa:=\Wak{\cri \Lam_0}^{\res}.
\end{align*}
The module $\resWa$
has a natural vertex algebra structure,
and  $\resWa\cong \D^{\ch}(\C^{\dim \finn})$
as vertex algebras.
There is a vertex algebra homomorphism
$\Vg{\cri}\ra \resWa$,
which 
induces the vertex algebra homomorphism
\begin{align}
\omega: \resV \ra \resWa
\label{eq:affine-to-restricted-Wakimoto},
\end{align}
see \cite{Fre05}.
The map  $\omega$ is injective 
because  $\resV$ is simple.

The fact that
$\resV$ is a vertex subalgebra of
$\resWa$
implies that
$d(z)$ can be considered as   a field on $\resWa$.
Hence
$\BRS{\bullet}{\resWa}$ is also naturally a vertex algebra.
Thus  applying the functor $\BRS{0}{?}$  
to (\ref{eq:affine-to-restricted-Wakimoto}),
we obtain the vertex algebra homomorphism
\begin{align*}
\omega_\W: \resW=\BRS{0}{\resV}\ra \BRS{0}{\resWa}.
\end{align*}

\begin{Pro}\label{Pro:injective restricted}
The vertex algebra homomorphism $\omega_{\W}
$
is injective.
In fact it induces an injective homomorphism 
$\gr \resW
\hookrightarrow \gr \BRS{0}{\resWa}
$
of vertex Poisson algebras.
\end{Pro}

In order to prove Proposition \ref{Pro:injective restricted}
we first describe the homomorphism
\begin{align*}
 \bar \omega: \gr \resV\ra \gr \resWa
\end{align*}
induced by $\omega$.
Recall that $\gr \resV\cong \C[\Nil_{\infty}]$
and $R_{\resV}\cong \C[\Nil]$, see (\ref{eq:assocaited-graded-of-critical-simple}).
 Let $\mc{B}$ be the set of Borel subalgebras in  $\fing$,
or the flag variety of $\fing$.
Denote by
 $U$ be the big cell,
i.e.,
the unique open $N$-orbit in $\mc{B}$,
where 
$N$ is the unipotent subgroup of $G$
corresponding to $\finn$.
Let $T^*\mc{B}$ be the cotangent bundle
of $\mc{B}$,
$\pi: T^*\mc{B}\ra \mc{B}$ the projection,
\begin{align*}
 \tilde{U}=\pi\inv(U).
\end{align*}
By construction \cite{Fre05}
we have
\begin{align}
R_{\resWa}\cong \C[\tilde{U}],
\quad \gr \resWa \cong (R_{\resWa})_{\infty}=\C[\tilde{U}_{\infty}],
\label{eq:a.g.restricited-wakimoto}
\end{align}
and
the homomorphism
\begin{align*}
\bar \omega_{|R_{\resV}}
:R_{\resV}\ra  R_{\resWa} 
\end{align*}
may be identified with
the
restriction $\tilde{U}\ra \mc{N}$
of the Springer resolution 
\begin{align*}
 \mu: T^* \mc{B}\ra \Nil.
\end{align*}
This in particular shows that
$\bar \omega$ 
is also injective.
Indeed, 
$\bar \omega_{|R_{\resV}}$
is injective because
it
is the composition of the isomorphism 
$\mu^*:\C[\mc{N}]\isomap  \Gamma(T^*\mc{B},\mc{O}_{T^*\mc{B}})$,
with
 the restriction map
$\Gamma(T^*\mc{B}, \mc{O}_{T^* \mc{B}})\ra 
\Gamma(\tilde{U},\mc{O}_{T^*\mc{B}})=\C[\tilde{U}]$.
Hence
it induces an injection 
$(R_{\resV})_{\infty}\hookrightarrow (R_{\resWa})_{\infty}$,
and this is identical to 
$\bar \omega$.

\begin{Rem}\label{Rem:cdo}
Let $\D_{\mc{B}}^{\ch}$
be the sheaf of {\em chiral differential operators}
\cite{GorMalSch04,MalSchVai99,BeiDri04}
 on $\mc{B}$,
which exists uniquely \cite{GorMalSch01,ArkGai02}.
It is a sheaf of vertex algebras
on $\mc{B}$,
and we have 
\begin{align}
 R_{\D_{\mc{B}}^{\ch}}
\cong \pi_* \mc{O}_{T^* \mc{B}},
\quad \gr \D_{\mc{B}}^{\ch}
\cong (\pi_{\infty})_* \mc{O}_{(T^*\mc{B})_{\infty}},
\label{eq;associated-graded-of-cdo}
\end{align}
where
$R_{\D_{\mc{B}}^{\ch}}$
and $\gr \D_{\mc{B}}^{\ch}$ are
the corresponding sheaves of Zhu's Poisson algebras
 respectively,
and
$\pi_{\infty}: (T^*\mc{B})_{\infty}\ra \mc{B}$
is the projection.
 We have
\begin{align}
\D_\mc{B}^{\ch}(U)\cong 
\resWa
\label{eq:Wakimovo-vs-CDO}
\end{align} as vertex algebras.
The homomorphism
(\ref{eq:affine-to-restricted-Wakimoto})
lifts to 
a vertex algebra homomorphism
\begin{align}
\omega_{\res}: \resV\ra \Gamma (\mc{B},\D_{\mc{B}}^{\ch}),
\end{align}
which is in fact an {\em isomorphism} \cite{AraCheMal08}.
\end{Rem}

Next we describe the
vertex Poisson algebra
structure of 
$\gr \BRS{0}{\resWa}$.
Let
\begin{align*}
\wSl=\mu\inv(\mc{S}),
\end{align*}
the  {\em Slodowy variety}.
It is known 
\cite{Gin08}
that
$\wSl$
 is a smooth, connected symplectic submanifold
of $T^* \mc{B}$
and
the morphism
$\mu_{|\wSl}: \wSl\ra \mc{S}$ is a symplectic resolution
of singularities.
As explained in \cite{Gin08},
$\wSl$
can 
be also obtained by means of the Hamiltonian reduction:
Let 
\begin{align*}
\tilde{\mu}: T^*\mc{B}\ra \finm^*
\end{align*}
be the composition of $T^* \mc{B}\overset{\mu}{\ra }\Nil
\hookrightarrow\fing^*$
with the restriction map $\fing^*\ra \finm^*$.
Then $\tilde{\mu}$
is  the moment map for the $M$-action
and the one point $M$-orbit $f\in \finm^*$
is a regular value of $\tilde{\mu}$.
Let
\begin{align*}
\Sigma=\tilde\mu\inv(f).
\end{align*}
Then 
$\Sigma$ 
is a
reduced smooth connected submanifold of $T^* \mc{B}$,
and the action map
gives an isomorphism
\begin{align}
 M \times \wSl\isomap \Sigma,
\label{eq:freeness-of-Slodowy-vareity}
\end{align}
and we get that  
\begin{align}
\wSl\cong \Sigma/M.
\end{align}
By (\ref{eq:freeness-of-Slodowy-vareity}),
we obtain the jet scheme analogue
 \begin{align}
  M_{\infty}\times S_{\infty}\isomap \Sigma_{\infty}.
\label{eq:freeness-of-Slodowy-vareity-jets}
 \end{align}

Let
\begin{align*}
 V=
\wSl\cap \tilde{U}.
\end{align*}
Because $\tilde U$
is $M$-stable,
by restricting (\ref{eq:freeness-of-Slodowy-vareity})
and (\ref{eq:freeness-of-Slodowy-vareity-jets})
we obtain the isomorphisms
\begin{align}
& M\times V\isomap \Sigma\cap \tilde{U}
,\\
&
M_{\infty}\times V_{\infty}\isomap \Sigma_{\infty}\cap
 \tilde{U}_{\infty}=(\Sigma\cap \tilde{U})_{\infty}.
\label{eq:iso-jets}
\end{align}
Note that 
$\Sigma\cap \tilde{U}$ is an open dense subset
of $\Sigma$,
and 
$V$ is an open dense subset of $\tilde{S}$.

 \begin{Pro}\label{Pro:vanishing-of-Wakimoto}
We have $\BRS{i}{\resWa}=0$
 for $i\ne 0$,
$R_{\BRS{0}{\resWa}}\cong 
\C[V]$
and $\gr {\BRS{0}{\resWa}}\cong \C[V_{\infty}]$.
\end{Pro}
\begin{proof}
 By \cite[3.7]{Ara05},
the differential $d$ decomposes as 
 $d=d^{st}+d^{\chi}$
with $(d^{st})^2=(d^{\chi})^2=\{d^{st},d^{\chi}\}=0$.
It follows that
 there is a 
spectral sequence
$E_r\Rightarrow 
\BRS{0}{\resWa}$ such that
$d_0=d^{st}$
and $d_1=d^{\chi}$.
By \cite[Remark 3.7.1]{Ara05}
we have
\begin{align}
H^{\bullet}(C(\resWa),d^{st})\cong  H^{\semiinf+\bullet}(L\fing_{>0},
\resWa\* S(\fing_{1/2}[t\inv]t\inv)
),
\label{eq:standard-cohomollgy1}
\end{align}where 
$L\fing_{>0}=\fing[t,t\inv]$
and
$S(\fing_{1/2}[t\inv]t\inv)$
is considered as
a 
$L\fing_{>0}$-module 
through  the identification
$S(\fing_{1/2}[t\inv]t\inv)\cong U(L\fing_{>0})/
U(L\fing_{>0})(L\fing_{\geq 1}+\fing_{1/2}[t])$.
Here,
$L\fing_{\geq 1}=\bigoplus_{j\geq 1}\fing_j[t,t\inv]$.

Because $\resWa$ is  free over $\finn[t\inv ]t\inv$
and cofree over $\finn[t]$ by construction,
it follows by
\cite[Theorem 2.1]{Vor93} that
$ H^{\semiinf+i}(L\fing_{>0},
\resWa\* S(\fing_{1/2}[t\inv]t\inv)
)
=0$ for $i\ne 0$.
Hence the spectral sequence collapses at $E_1=E_{\infty}$
and we get that
\begin{align}
 \BRS{i}{\resWa}\cong \begin{cases}
		  H^0(C(\resWa),d^{st})&\text{for $i=0$},\\0
&
\text{for $i\ne 0$.}
		 \end{cases}
\label{eq:standard-cohomollgy2}
\end{align}
This proves the first assertion.

Let $\gr C(\resWa)$ be the 
associated graded complex of 
$C(\resWa)$ with respect to the Li filtration of $C(\resWa)$.
Then
as in \cite[Theorem 4.3.3]{Ara09b}
we find that
\begin{align}
& H^i(\gr C(\resWa))\cong \begin{cases}
		      (\gr \resWa/I_{\infty}\gr \resWa )^{M_{\infty}}
\cong \C[V_{\infty}]
&
(i=0)\\ 0&(i\ne 0).
		     \end{cases}.
\label{eq:vanishing-associated-graded}
\end{align}
By 
(\ref{eq:vanishing-associated-graded}),
 (\ref{eq:standard-cohomollgy2})
and
\cite[Proposition 4.4.3]{Ara09b},
the spectral sequence associated with the Li filtration of $C(\resWa)$
converges to $\BRS{\bullet}{\resWa}$.
Hence we have
$
 \gr^K \BRS{0}{\resWa}
\cong \C[V_{\infty}]
$,
where
$\gr^K \BRS{0}{\resWa}$
is the associated graded vertex algebra
with respect to the filtration
$\{K^p \BRS{0}{\resWa}\}$ 
induced by the Li filtration of 
$C(\resWa)$.
As in the proof of Proposition \ref{Pro:description-of-restricted-W},
we see that
$\{K^p \BRS{0}{\resWa}\}$ 
coincides wit the Li filtration of $\BRS{0}{\resWa}$.
This proves the last assertion,
which restricts to the second assertion.

\end{proof}

\begin{proof}[Proof of Proposition \ref{Pro:injective restricted}]
Let
$\mu_{|V}
: V\ra S$
be the restriction
of the resolution $\mu_{|\wSl}:\wSl\ra \mc{S}$.
By Propositions \ref{Pro:description-of-restricted-W}
and
\ref{Pro:vanishing-of-Wakimoto},
\begin{align}
\mu^*_{|V}: \C[\mc{S}]\ra \C[V]
\label{eq:1}
\end{align}
can be identified with  the 
 homomorphism
$R_{\resW}\ra R_{\BRS{0}{\resWa}}$
induced by the vertex algebra homomorphism  $
\omega_\W:\resW\ra \BRS{0}{\resWa}$.
Thus,
by Propositions \ref{Pro:description-of-restricted-W}
and \ref{Pro:vanishing-of-Wakimoto},
it is sufficient to show that
(\ref{eq:1}) is injective.
But
(\ref{eq:1}) 
is the composition
of
$\mu^*: \C[\mc{S}]\isomap \Gamma(\wSl,\mc{O}_{\wSl})$
with
 the 
restriction map
$\Gamma(\wSl,\mc{O}_{\wSl})\ra 
\C[V]$.
Hence it is injective as required.
\end{proof}

\section{Proof of Theorems \ref{MainTh} and \ref{Th: main2}}
For $\lam\in \dual{\affh}_k$,
let $\M{\lam}$ be the 
Verma module with highest weight $\lam$,
$\M{\lam}^*$ its  contragredient     dual.

 \begin{Pro}\label{Pro:cocyclicity}
Suppose that
$\BRS{i}{\M{\lam}^*}=0$
for $i\ne 0$.
Then
  $\BRS{0}{\M{\lam}^*}$ is a
cocyclic $\Wg{k}$-module  with the cocyclic vector $v_{\lam}^*$,
where $v_{\lam}^*$
is
the image of  the cocyclic vector of $\M{\lam}^*$.
\end{Pro}
\begin{proof}[Proof (outline)]
 By the argument of \cite[\S 6]{Ara04},
\cite[\S 7]{Ara05}, \cite[\S 7]{Ara07},
we can construct a subcomplex $C'$
of $C(\M{\lam}^*)$
with the following properties:
\begin{enumerate}
 \item $H^i(C')=0$ for $i\ne 0$;
\item  $C'$ is a $\Wg{k}$-submodule of $C(\M{\lam}^*)$
containing $v_{\lam}^*$,
and moreover,
$H^0(C')$ is a cocycolic $\Wg{k}$-module
with the cocyclic vector $v_{\lam}^*$:
\item The character of $H^0(C')$ coincides with
the character of  $\BRS{0}{\M{\lam}^*}$.
\end{enumerate}

Because $H^0(C')$ is cocyclic 
the above property (iii) forces that 
the map 
$H^0(C')\ra \BRS{0}{\M{\lam}^*}$ is an injection.
But 
$H^0(C')$ and  $ \BRS{0}{\M{\lam}^*}$  
have the same character.
Therefore it must be an isomorphism.
\end{proof}

Let $\Delta_+$
be the
set of positive roots of $\fing$,
$W$  the Weyl group of $\fing$,
$\rho=\sum_{\alpha\in \Delta_+}\alpha/2$,
$\rho\che=\sum_{\alpha\in \Delta_+}\alpha\che/2$.
Denote by
 $\wh{\Delta}^{re}_+$  the set of positive real roots of $\bigaffg$.
The set $\Delta_+$ is naturally considered as a subset
of $\wh{\Delta}_+^{re}$.
Let
$\wt{W}=W\ltimes
P\che
$
the extended affine Weyl group of $\bigaffg$,
where
$P\che\subset \finh$
is the set of coweights of $\fing$.
We denote by $t_{\mu}$
the element of $\wt{W}$
corresponding $\mu\in P\che$.
Set 
\begin{align*}
\affP^+=\{\lam\in \dual{\affh}_k;
\lam(\alpha\che)\in \Z_{\geq 0}
\text{ for all }\alpha\in \Delta_+\}.
\end{align*}
\begin{Pro}\label{Pro:Wakimoto-vs-Verma}
Suppose that
$k+h\che\not \in \Q_{>0}$.
Then,
for  $\lam\in \affP^+$,
$\M{\lam}^*$
is free over 
$\finn[t\inv]t\inv$.
\end{Pro}
\begin{proof}
By the assumption
\begin{align}
\bra \lam+\widehat \rho,\alpha\che\ket \not \in \N
\quad \text{for $\alpha\in \wh \Delta^{re}_+$
such that $\bar \alpha\in -\Delta_+$},
\end{align}
where $\widehat{\rho}=\rho+h\che \Lam_0$,
Hence 
$\bra \lam+\widehat \rho,\alpha\che\ket \not \in \N$
for all $\alpha\in 
\wh{\Delta}^{re}_+ \cap t_{-n\rho\che}
(-\wh{\Delta}^{re}_+)$,  $n\in \N$.
By
 \cite[Theorem 3.1]{Ara04},
this implies that
$\M{\lam}$ is 
cofree over the subalgebra 
$\bigoplus\limits_{\alpha\in \wh{\Delta}^{re}_+ 
\cap t_{-n\rho\che}(-\wh{\Delta}^{re}_+)}\affg_{\alpha}\subset
 \finn_-[t]t$,
or equivalently,
$\M{\lam}^*$ is free over 
$\bigoplus\limits_{\alpha\in (-\wh{\Delta}^{re}_+ )
\cap t_{n\rho\che}(\wh{\Delta}^{re}_+)}\affg_{\alpha}\subset
 \finn[t\inv]t\inv$.
Here $\affg_{\alpha}$
is the root space of $\affg$ of root $\alpha$.
Now we have 
$\finn[t\inv]t\inv=\lim\limits_{\ra\atop n}
\bigoplus\limits_{\alpha\in -\wh{\Delta}^{re}_+
\cap t_{n\rho\che}(\wh{\Delta}^{re}_+)}\affg_{\alpha}$.
Therefore
$\M{\lam}^*$ is free over $\finn[t\inv]t\inv$ as required.
\end{proof}

\begin{Rem}
 Proposition \ref{Pro:Wakimoto-vs-Verma} implies that
$\M{\lam}^*$
with $\lam\in \affP^+$,
$k+h\che\not\in \Q_{\geq 0}$,
is isomorphic to
the Wakimoto module 
$\Wak{\lam}$
with highest weight $\lam$,
see the proof of Proposition \ref{Pro:Wakimoto is Verma}.
\end{Rem}

\begin{Pro}\label{Pro:cocyclicity-dual-Verma}
 Suppose that $k+h\che\not \in \Q_{>0}$
and let
 $\lam \in \affP^+$.
Then
we have
$\BRS{i}{\M{\lam}^*}=0$  for $i\ne 0$
and $\BRS{0}{\M{\lam}^*}$ is cocyclic $\Wg{k}$-module
with the cocyclic vector $v_\lam^*$.
\end{Pro}
\begin{proof}

$\M{\lam}^*$ is 
cofree over  $\finn[t]$ by definition,
and  
 is free over $\finn[t\inv]t\inv$
by Proposition \ref{Pro:Wakimoto-vs-Verma}.
Hence one can apply the proof of Proposition  \ref{Pro:vanishing-of-Wakimoto}
to obtain the vanishing assertion.
The rest follows from Proposition \ref{Pro:cocyclicity}.
\end{proof}

Now let $k=\cri$.
For $\lam\in \dual{\affh}_{\cri}$,
let $\M{\lam}^{\res}$ be the 
{\em restricted Verma module} \cite{FeiFre90,AraFie08}
with highest weight $\lam$,
$(\M{\lam}^{\res})^*$ its  contragredient  dual.
The module $\M{\lam}^{\res}$ is defined as follows:
Let 
$p^{(1)},\dots ,p^{(l)}$
be  a set of homogeneous generators of $\mf{z}(\affg)$
as a differential algebra (which is the same as a commutative vertex algebra),
so that $R_{\mf{z}(\affg)}=\C[\bar p^{(1)},\dots,\bar p^{(l)}]$,
where $\bar p^{(i)}$ is  the image of $p^{(i)}$ in
$R_{\mf{z}(\affg)}=\C[\fing^*]^G$.
Let $Y(p_i,z)=\sum_{n\in \Z}p^{(i)}_n
 z^{-n-\Delta_i}$
be the field corresponding to $p^{(i)}$,
where $\Delta_i$ is the degree of the polynomial $\bar p^{(i)}$.
Then
\begin{align*}
 Z_{\pm}=\C[p^{(i)}_n;
i=1,\dots,l, \pm n>0]
\end{align*}
can be regarded as a polynomial ring,
which acts on 
any $\Vg{\cri}$-module.
According to Feigin and Frenkel \cite{FeiFre90,Fre07},
$Z_-$ acts on $\M{\lam}$  freely.
By definition,
\begin{align*}
\M{\lam}^{\res}=\M{\lam}/Z_-^* \M{\lam},
\end{align*}
where $Z_-^*$ is the argumentation ideal of $Z_-$.
Dually,
\begin{align*}
(\M{\lam}^{\res})^*=\{m\in \M{\lam}^*;
Z_+^* m=0\},
\end{align*}where
 $Z_+^*$ is the argumentation ideal of $Z_+$.

\begin{Pro}\label{Pro cocyclisity}
Let $\lam\in \widehat{P}_{\cri}^+$. 
\begin{enumerate}
 \item The embedding
$(\M{\lam}^{\res})^*\hookrightarrow 
\M{\lam}^*$ induces an embedding
$\BRS{0}{(\M{\lam}^{\res})^*}\hookrightarrow 
\BRS{0}{\M{\lam}^*}$.
\item
The
$\Wg{\cri}$-module
$\BRS{0}{(\M{\lam}^{\res})^*} $
is a cocyclic 
 with the cocyclic  vector $v_{\lam}^*$.
\end{enumerate}
\end{Pro}
\begin{proof}
(1)
Let 
$\Gamma^0 \M{\lam}^*=0$,
$\Gamma^p \M{\lam}^*=\{m\in \M{\lam}^*; (Z_+^*)^ pm=0\}$ for $p\geq 1$.
Then 
$\{\Gamma^p \M{\lam}^*\}$ defines an increasing filtration of
$\M{\lam}^*$ as a $\affg$-module,
and 
the freeness of $\M{\lam}$ over $Z_-$ implies that
 \begin{align*}
  \gr^{\Gamma}\M{\lam}^*\cong (\M{\lam}^{\res})^*\* D(Z_+),
 \end{align*}
as $\Vg{k}\* Z_+$-modules,
where $D(Z_+)$ is the restricted dual of $Z_+$.

Now 
$(\M{\lam}^{\res})^*$ is free over $\finn[t\inv]t\inv$
and cofree over $\finn[t]$
by Proposition \ref{Pro:Wakimoto is Verma}.
Thus
we see as in the proof of Proposition \ref{Pro:vanishing-of-Wakimoto}
that
$\BRS{i}{(\M{\lam}^{\res})^*}=0$ for $i\ne 0$.
This shows that
the spectral sequence corresponding to 
the filtration ${\Gamma^p\M{\lam}^*}$
collapses at $E_1=E_{\infty}$
and we get that
\begin{align*}
\gr^{\Gamma} \BRS{0}{\M{\lam}^*}\cong \BRS{0}{(\M{\lam}^{\res})^*}\*
D(Z_+).
\end{align*}
In particular,
\begin{align*}
 \BRS{0}{(\M{\lam}^{\res})^*}\cong 
\Gamma^1 \BRS{0}{\M{\lam}^*}=\{
c\in \BRS{0}{\M{\lam}^*}; Z_+^* c=0\}.
\end{align*}
This proves (1).
(2) follows from (1)
and
 Proposition \ref{Pro:cocyclicity-dual-Verma}.
\end{proof}

\begin{Pro}\label{Pro:Wakimoto is Verma}
For $\lam \in \widehat{P}_{\cri}^+$,
the restricted Wakimoto module
$\Wak{\lam}^{\res}$ is isomorphic to $ (\M{\lam}^{\res})^*$.
\end{Pro}
\begin{proof}
By \cite[6.2.2]{AraCheMal08},
$\Wak{\lam}^{\res}$ is  cocyclic with the cocyclic vector $|\lam\ket$,
where 
$|\lam\ket$ is the highest weight vector.
Hence its contragredient   dual
$(\Wak{\lam}^{\res})^*$
is cyclic, and the natural 
$\affg$-module homomorphism
$\M{\lam}\ra (\Wak{\lam}^{\res})^*$
 is surjective.
Because $Z_-$ acts trivially on
$(\Wak{\lam}^{\res})^*$, this factors though the surjective homomorphism
$\M{\lam}^{\res}\ra (\Wak{\lam}^{\res})^*$.
Since 
$(\Wak{\lam}^{\res})^*$ and
$\M{\lam}^{\res}$ are the same
character,
 it must be an isomorphism.
By  duality, this proves the assertion.
\end{proof}

\begin{proof}[Proof of Theorem \ref{Th: main2}]
We have already shown
the assertion on the associated graded vertex Poisson algebras in
Proposition 
\ref{Pro:conforma-weight-filtartion}.
It remains to prove the simplicity.

First, let
$\chi=\chi_0$.
By  
Propositions \ref{Pro:injective restricted}
and
\ref{Pro:Wakimoto is Verma},
$\resW$
 is a submodule of
$\BRS{0}{(\M{\cri\Lam_0}^{\res})^*}$.
On the other hand,
$\BRS{0}{(\M{\cri\Lam_0}^{\res})^*}$
is
 cocyclic
by Proposition \ref{Pro cocyclisity},
and the image of 
the vacuum vector $\1$ of $\resW$
equals to
the  cocyclic vector
of $\BRS{0}{(\M{\cri\Lam_0}^{\res})^*}$
up to nonzero constant multiplication.
Hence $\resW$ is also cocyclic, with the cocyclic vector $\1$.
  Therefore   $\resW$ must be  simple.

Next, let $\chi$ be arbitrary.
Let $\{E_p \criW{\chi}\}$ be the conformal filtration
of $\criW{\chi}$
as in the proof of Proposition \ref{Pro:conforma-weight-filtartion}.
Then \eqref{eq:conformal-filtration--}
shows that
$\gr^E \criW{\chi}\cong \resW$ as $\Wg{\cri}$-modules,
which is simple.
Therefore $\criW{\chi}$ is also simple.
This completes the proof.
\end{proof}

\begin{proof}[Proof of Theorem \ref{MainTh}]
The first  assertion follows immediately from 
Proposition \ref{Pro:injectivity}
and Theorem \ref{Th: main2}.
The freeness assertion 
has been  proved in Proposition \ref{Pro:description-of-restricted-W}.
 \end{proof}
\bibliographystyle{jalpha}
\bibliography{math}

\end{document}